\documentclass[11pt,reqno]{amsart}
\usepackage{amsfonts,amssymb,amsmath,color}

\newtheorem{theorem}{Theorem}

\newtheorem{corollary}[theorem]{Corollary}
\newtheorem{lemma}[theorem]{Lemma}
\newtheorem{proposition}[theorem]{Proposition}

\newtheorem{remark}[theorem]{Remark}

\numberwithin{equation}{section} \numberwithin{theorem}{section}

\begin{document}

\title{Pattern Recognition on Oriented Matroids: Symmetric Cycles in the Hypercube Graphs. II}

\author{Andrey O. Matveev}
\email{andrey.o.matveev@gmail.com}

\begin{abstract}
We present explicit descriptions of the decompositions of vertices of a hypercube graph with respect to its distinguished symmetric cycle.
\end{abstract}

\maketitle

\pagestyle{myheadings}

\markboth{PATTERN RECOGNITION ON ORIENTED MATROIDS}{A.O.~MATVEEV}

\thispagestyle{empty}


\section{Introduction}

Let $\boldsymbol{H}(t,2)$ be the {\em hypercube graph\/} with its vertex set $\{1,-1\}^t$ composed of {\em row vectors\/} $T:=(T(1),\ldots,T(t))$
of the real Euclidean space $\mathbb{R}^t$ with $t\geq 3$. Vertices~$T'$ and~$T''$ are adjacent in $\boldsymbol{H}(t,2)$ if and only if there is a unique element~$e\in E_t:=[t]:=[1,t]:=\{1,\ldots,t\}$ such that $T'(e)=-T''(e)$.

The graph $\boldsymbol{H}(t,2)$ can be regarded as the {\em tope graph\/} of the {\em oriented matroid\/} $\mathcal{H}:=(E_t,\{1,-1\}^t)$ on the {\em ground set\/} $E_t$, with the set of {\em topes\/}~$\{1,-1\}^t$; this oriented matroid is realizable as the {\em arrangement\/} of {\em coordinate hyperplanes\/} in $\mathbb{R}^t$, see, e.g.,~\cite[Example~4.1.4]{BLSWZ}.

We denote by $\mathrm{T}^{(+)}:=(1,\ldots,1)$ the {\em positive tope\/} of $\mathcal{H}$; the {\em negative tope\/} is the tope $\mathrm{T}^{(-)}:=-\mathrm{T}^{(+)}$. If $T\in\{1,-1\}^t$ and $A\subseteq E_t$, then the tope~${}_{-A}T$ by definition has the components
\begin{equation*}
({}_{-A}T)(e):=\begin{cases}
\phantom{-}1\; , & \text{if $e\in A$ and $T(e)=-1$}\; ,\\
-1\; , & \text{if $e\in A$ and $T(e)=\phantom{-}1$}\; ,\\
T(e)\; , & \text{otherwise}\; .
\end{cases}
\end{equation*}
If $\{a\}$ is a one-element subset of $E_t$, then we write ${}_{-a}T$ instead of ${}_{-\{a\}}T$.

In this note, we describe explicitly the {\em decompositions\/} of topes~\mbox{$T\in\{1,-1\}^t$} with respect to one distinguished {\em symmetric $2t$-cycle\/} $\boldsymbol{R}:=(R^0,R^1,\ldots,R^{2t-1},$ $R^0)$
in $\boldsymbol{H}(t,2)$, which is defined as follows:
\begin{equation}
\label{eq:9}
\begin{split}
R^0:\!&=\mathrm{T}^{(+)}\; ,\\
R^s:\!&={}_{-[s]}R^0\; ,\ \ \ 1\leq s\leq t-1\; ,
\end{split}
\end{equation}
and
\begin{equation}
\label{eq:10}
R^{k+t}:=-R^k\; ,\ \ \ 0\leq k\leq t-1\; .
\end{equation}

Consider the nonsingular matrix
\begin{equation}
\label{eq:26}
\mathbf{M}:=\mathbf{M}(\boldsymbol{R}):=\left(\begin{smallmatrix}
R^0\\ R^1\\ \vdots\\ R^{t-1}
\end{smallmatrix}
\right)\in\mathbb{R}^{t\times t}
\end{equation}
whose rows are the topes given in~(\ref{eq:9}). The $i$th row $(\mathbf{M}^{-1})_i$, $1\leq i\leq t$, of the inverse matrix $\mathbf{M}^{-1}$ of $\mathbf{M}$ is
\begin{equation*}
(\mathbf{M}^{-1})_i=\begin{cases}
\frac{1}{2}\cdot(\,\boldsymbol{\sigma}(i)-\boldsymbol{\sigma}(i+1)\,)\; , & \text{if $i\neq t$}\; ,\\
\frac{1}{2}\cdot(\,\boldsymbol{\sigma}(1)+\boldsymbol{\sigma}(t)\,)\; , & \text{if $i=t$}\; ,
\end{cases}
\end{equation*}
where
$\boldsymbol{\sigma}(s):=(0,\ldots,\underset{\overset{\uparrow}{s}}{1},\ldots,0)$
are vectors of the standard basis of $\mathbb{R}^t$.

Recall that for any tope $T\in\{1,-1\}^t$ there exists a unique {\em inclusion-minimal\/} subset~$\boldsymbol{Q}(T,\boldsymbol{R})$ of the vertex set of the cycle~$\boldsymbol{R}$ such that
\begin{equation*}
T=\sum_{Q\in\boldsymbol{Q}(T,\boldsymbol{R})}Q\; ;
\end{equation*}
see~\cite[Sect.~11.1]{AM-PROM-I}.

If we define the row vector $\boldsymbol{x}:=\boldsymbol{x}(T):=\boldsymbol{x}(T,\boldsymbol{R})\in\{-1,0,1\}^t$ by
\begin{equation*}
\boldsymbol{x}
:=T\mathbf{M}^{-1}\; ,
\end{equation*}
then
\begin{equation*}
\boldsymbol{Q}(T,\boldsymbol{R})=\{x_i\cdot R^{i-1}\colon x_i\neq 0\}
\end{equation*}
and
\begin{equation*}
|\boldsymbol{Q}(T,\boldsymbol{R})|=|\{i\in E_t\colon x_i\neq 0\}|=\|\boldsymbol{x}(T)\|^2\; ,
\end{equation*}
where $\|\boldsymbol{x}\|^2:=\langle \boldsymbol{x},\boldsymbol{x}\rangle$, and $\langle\cdot,\cdot\rangle$ is the standard scalar product on $\mathbb{R}^t$.
We have $\|\boldsymbol{x}(T)\cdot\mathbf{M}\|^2=\|T\|^2=t$, and $\sum_{e\in E_t}x_e(T)=\langle\boldsymbol{x}(T),\mathrm{T}^{(+)}\rangle=T(t)$; we will see in Sect.~\ref{chap:11:1} that
if $x_t(T)\neq 0$, then $T(t)=x_t(T)$.

In Proposition~\ref{prop:3} of this note, we describe explicitly the vectors $\boldsymbol{x}(T,\boldsymbol{R})$ associated with
vertices $T$ of the hypercube graph $\boldsymbol{H}(t,2)$ and with its distinguished symmetric cycle $\boldsymbol{R}$ defined by~(\ref{eq:9})(\ref{eq:10}).

\section{Separation sets, the negative parts, and the
decompositions of vertices of the hypercube graph}
\label{chap:11:1}

If $T',T''\in\{1,-1\}^t$ are two vertices of the hypercube graph~$\boldsymbol{H}(t,2)$, then
\begin{equation}
\label{eq:19}
T''=T'-2\sum_{s\in\mathbf{S}(T',T'')}T'(s)\boldsymbol{\sigma}(s)\; ,
\end{equation}
where $\mathbf{S}(T',T''):=\{e\in E_t\colon T'(e)\neq T''(e)\}$ is the {\em separation set\/} of the topes~$T'$ and~$T''$. For the topes in the ordered collection~(\ref{eq:9}), we have \mbox{$\boldsymbol{x}(R^{s-1})=\boldsymbol{\sigma}(s)$}, $s\in E_t$.

Relation~(\ref{eq:19}) implies that
\begin{equation}
\label{eq:22}
\boldsymbol{x}(T'')=\boldsymbol{x}(T')-2\Bigl(\sum_{s\in\mathbf{S}(T',T'')}T'(s)\boldsymbol{\sigma}(s)\Bigr)\cdot\mathbf{M}^{-1}\; .
\end{equation}
In particular, for any
vertex $T\in\{1,-1\}^t$ of $\boldsymbol{H}(t,2)$, we have
\begin{equation}
\label{eq:15}
\begin{split}
\boldsymbol{x}(T)&=\boldsymbol{x}(\mathrm{T}^{(+)})-2\Bigl(\sum_{s\in T^-}\boldsymbol{\sigma}(s)\Bigr)\cdot\mathbf{M}^{-1}\\
&=\boldsymbol{\sigma}(1)-2\Bigl(\sum_{s\in T^-}\boldsymbol{\sigma}(s)\Bigr)\cdot\mathbf{M}^{-1}\; ,
\end{split}
\end{equation}
where $T^-:=\{e\in E_t\colon T(e)=-1\}$ is the {\em negative part\/} of the tope~$T$.

Since
\begin{equation*}
|T^-|=\frac{1}{2}\bigl(t-\langle T,\mathrm{T}^{(+)}\rangle\bigr)=\frac{1}{2}\bigl(t-\sum_{e\in E_t}T(e)\bigr)\; ,
\end{equation*}
we have
\begin{equation*}
|T^-|=j
\end{equation*}
for some integer $j$ if and only if
\begin{equation*}
\begin{split}
\langle T,\mathrm{T}^{(+)}\rangle
&=\boldsymbol{x}(T)\cdot\mathbf{M}\cdot(\mathrm{T}^{(+)})^{\top}\\&=\boldsymbol{x}(T)\cdot(t,t-2,t-4,\ldots,-(t-2))^{\top}=t-2j\; .
\end{split}
\end{equation*}

Note also that for any two vertices $T'$ and $T''$ of $\boldsymbol{H}(t,2)$ we have
\begin{align*}
|(T')^- \cap (T'')^-|&=\tfrac{1}{4}\langle T'-\mathrm{T}^{(+)},T''-\mathrm{T}^{(+)}\rangle\\&=
\tfrac{1}{4}\bigl(t+\langle T',T''\rangle-\langle T'+T'',\mathrm{T}^{(+)}\rangle\bigr)\\
\intertext{and}
|(T')^- \cup (T'')^-|&=
\tfrac{1}{4}\bigl(3t-\langle T',T''\rangle-\langle T'+T'',\mathrm{T}^{(+)}\rangle\bigr)\; .
\end{align*}

\begin{remark}
For the symmetric cycle $\boldsymbol{R}$ in the the hypercube graph~$\boldsymbol{H}(t,2)$, defined by~{\rm(\ref{eq:9})(\ref{eq:10})}, let $\boldsymbol{x},\boldsymbol{x}',\boldsymbol{x}''\in\{-1,0,1\}^t$ be row vectors such that the row vectors $T:=\boldsymbol{x}\cdot\mathbf{M}(\boldsymbol{R})$, $T':=\boldsymbol{x}'\cdot\mathbf{M}(\boldsymbol{R})$ and $T'':=\boldsymbol{x}''\cdot\mathbf{M}(\boldsymbol{R})$ are vertices of $\boldsymbol{H}(t,2)$.
\begin{itemize}
\item[\rm(i)]
We have
\begin{equation*}
\begin{split}
|T^-|&=\frac{1}{2}\bigl(t-\sum_{i\in E_t} x_i\cdot(t-2(i-1))\bigr)=\frac{t}{2}\bigl(1-\sum_{e\in E_t}x_e\bigr)+\sum_{i\in[2,t]}x_i\cdot(i-1)\\
&=\begin{cases}
t+\sum_{i\in[2,t]}x_i\cdot(i-1)\; , & \text{if $\langle\boldsymbol{x},\mathrm{T}^{(+)}\rangle=-1$}\; ,\\
\phantom{t+}\;\sum_{i\in[2,t]}x_i\cdot(i-1)\; , & \text{if $\langle\boldsymbol{x},\mathrm{T}^{(+)}\rangle=\phantom{-}1$}
\end{cases}\\
&=\begin{cases}
t+1+\sum_{i\in E_t}x_i\cdot i\; , & \text{if $\langle\boldsymbol{x},\mathrm{T}^{(+)}\rangle=-1$}\; ,\\
\phantom{t}-1+\sum_{i\in E_t}x_i\cdot i\; , & \text{if $\langle\boldsymbol{x},\mathrm{T}^{(+)}\rangle=\phantom{-}1$}\; .
\end{cases}
\end{split}
\end{equation*}

\item[\rm(ii)]
We have
\begin{multline*}
|(T')^- \cap (T'')^-|=
\tfrac{1}{4}\bigl(t+\langle \boldsymbol{x}'\cdot\mathbf{M},\boldsymbol{x}''\cdot\mathbf{M}\rangle-\langle (\boldsymbol{x}'+\boldsymbol{x}'')\cdot\mathbf{M},\mathrm{T}^{(+)}\rangle\bigr)\\
=
\begin{cases}
\phantom{-}\frac{3t}{4}+1+\frac{1}{4}\langle \boldsymbol{x}'\cdot\mathbf{M},\boldsymbol{x}''\cdot\mathbf{M}\rangle+\frac{1}{2}
\sum_{i\in E_t}(x_i'+x_i'')\cdot i\; , & \quad\\
\ \ \ \ \ \ \ \ \ \ \ \ \ \ \ \ \ \ \ \ \ \ \ \ \ \ \ \ \ \ \ \ \ \text{if $\langle\boldsymbol{x}'',\mathrm{T}^{(+)}\rangle=\langle\boldsymbol{x}',\mathrm{T}^{(+)}\rangle=-1$}\; ,\\
\phantom{-}\frac{t}{4}+\frac{1}{4}\langle \boldsymbol{x}'\cdot\mathbf{M},\boldsymbol{x}''\cdot\mathbf{M}\rangle+\frac{1}{2}
\sum_{i\in E_t}(x_i'+x_i'')\cdot i\; , & \quad\\
\ \ \ \ \ \ \ \ \ \ \ \ \ \ \ \ \ \ \ \ \ \ \ \ \ \ \ \ \ \ \ \ \ \text{if $\langle\boldsymbol{x}'',\mathrm{T}^{(+)}\rangle=-\langle\boldsymbol{x}',\mathrm{T}^{(+)}\rangle$}\; ,\\
-\frac{t}{4}-1+\frac{1}{4}\langle \boldsymbol{x}'\cdot\mathbf{M},\boldsymbol{x}''\cdot\mathbf{M}\rangle+\frac{1}{2}
\sum_{i\in E_t}(x_i'+x_i'')\cdot i\; , & \quad\\
\ \ \ \ \ \ \ \ \ \ \ \ \ \ \ \ \ \ \ \ \ \ \ \ \ \ \ \ \ \ \ \ \ \text{if $\langle\boldsymbol{x}'',\mathrm{T}^{(+)}\rangle=\langle\boldsymbol{x}',\mathrm{T}^{(+)}\rangle=1$}
\end{cases}
\end{multline*}
and
\begin{multline*}
|(T')^- \cup (T'')^-|=
\tfrac{1}{4}\bigl(3t-\langle \boldsymbol{x}'\cdot\mathbf{M},\boldsymbol{x}''\cdot\mathbf{M}\rangle-\langle (\boldsymbol{x}'+\boldsymbol{x}'')\cdot\mathbf{M},\mathrm{T}^{(+)}\rangle\bigr)\\
=\begin{cases}
\phantom{-}\frac{5t}{4}+1-\frac{1}{4}\langle \boldsymbol{x}'\cdot\mathbf{M},\boldsymbol{x}''\cdot\mathbf{M}\rangle+\frac{1}{2}
\sum_{i\in E_t}(x_i'+x_i'')\cdot i\; , & \quad\\
\ \ \ \ \ \ \ \ \ \ \ \ \ \ \ \ \ \ \ \ \ \ \ \ \ \ \ \ \ \ \ \ \ \text{if $\langle\boldsymbol{x}'',\mathrm{T}^{(+)}\rangle=\langle\boldsymbol{x}',\mathrm{T}^{(+)}\rangle=-1$}\; ,\\
\phantom{-}\frac{3t}{4}-\frac{1}{4}\langle \boldsymbol{x}'\cdot\mathbf{M},\boldsymbol{x}''\cdot\mathbf{M}\rangle+\frac{1}{2}
\sum_{i\in E_t}(x_i'+x_i'')\cdot i\; , & \quad\\
\ \ \ \ \ \ \ \ \ \ \ \ \ \ \ \ \ \ \ \ \ \ \ \ \ \ \ \ \ \ \ \ \ \text{if $\langle\boldsymbol{x}'',\mathrm{T}^{(+)}\rangle=-\langle\boldsymbol{x}',\mathrm{T}^{(+)}\rangle$}\; ,\\
\phantom{-}\frac{t}{4}-1-\frac{1}{4}\langle \boldsymbol{x}'\cdot\mathbf{M},\boldsymbol{x}''\cdot\mathbf{M}\rangle+\frac{1}{2}
\sum_{i\in E_t}(x_i'+x_i'')\cdot i\; , & \quad\\
\ \ \ \ \ \ \ \ \ \ \ \ \ \ \ \ \ \ \ \ \ \ \ \ \ \ \ \ \ \ \ \ \ \text{if $\langle\boldsymbol{x}'',\mathrm{T}^{(+)}\rangle=\langle\boldsymbol{x}',\mathrm{T}^{(+)}\rangle=1$}\; .
\end{cases}
\end{multline*}
\end{itemize}
\end{remark}

Since
\begin{equation*}
\{1,-1\}^t=\Bigl\{\mathrm{T}^{(+)}-2\sum_{s\in A}\boldsymbol{\sigma}(s)\colon A\subseteq E_t\Bigr\}\; ,
\end{equation*}
we have
\begin{equation*}
\bigl\{\boldsymbol{x}(T)\colon T\in\{1,-1\}^t\bigr\}=
\Bigl\{\boldsymbol{\sigma}(1)-2\Bigl(\sum_{s\in A}\boldsymbol{\sigma}(s)\Bigr)\cdot\mathbf{M}^{-1}\colon A\subseteq E_t\Bigr\}\; .
\end{equation*}

If $s\in E_t$, then we define a row vector $\boldsymbol{y}(s):=\boldsymbol{y}(s;t)\in\{-1,0,1\}^t$ by
\begin{equation*}
\boldsymbol{y}(s):=\boldsymbol{x}({}_{-s}\mathrm{T}^{(+)})=\boldsymbol{\sigma}(1)-2\boldsymbol{\sigma}(s)\cdot\mathbf{M}^{-1}\; ,
\end{equation*}
that is,
\begin{equation*}
\boldsymbol{y}(s):=
\begin{cases}
\phantom{-}\boldsymbol{\sigma}(2)\; , & \text{if $s=1$}\; ,\\
\phantom{-}\boldsymbol{\sigma}(1) - \boldsymbol{\sigma}(s) + \boldsymbol{\sigma}(s+1)\; , & \text{if $1<s<t$}\; ,\\
-\boldsymbol{\sigma}(t)\; , & \text{if $s=t$}\; .
\end{cases}
\end{equation*}

\begin{remark}
\label{prop:1}
If~$A\subseteq E_t$, then
\begin{equation*}
\begin{split}
\boldsymbol{x}({}_{-A}\mathrm{T}^{(+)})&=-\boldsymbol{x}({}_{-(E_t-A)}\mathrm{T}^{(+)})\\&=(1-|A|)\cdot\boldsymbol{\sigma}(1)+\sum_{s\in A}\boldsymbol{y}(s)
\end{split}
\end{equation*}
and, as a consequence, we have
\begin{equation*}
\bigl\{\boldsymbol{x}(T)\colon T\in\{1,-1\}^t\bigr\}=
\Bigl\{(1-|A|)\cdot\boldsymbol{\sigma}(1)+\sum_{s\in A}\boldsymbol{y}(s)\colon A\subseteq E_t\Bigr\}\; .
\end{equation*}
\end{remark}

For subsets $A\subseteq E_t$, relations~(\ref{eq:15}) imply the following:
\begin{equation*}
\{1,t\}\cap A=\{1\}\ \ \ \Longrightarrow\ \ \ \boldsymbol{x}({}_{-A}\mathrm{T}^{(+)})=
\boldsymbol{\sigma}(1)-(\;\underbrace{\boldsymbol{\sigma}(1)-\boldsymbol{\sigma}(2)}_{2\cdot(\mathbf{M}^{-1})_1}\;)\;
-\sum_{i\in A-\{1\}}(\;\underbrace{\boldsymbol{\sigma}(i)-\boldsymbol{\sigma}(i+1)}_{2\cdot(\mathbf{M}^{-1})_i}\;)\; ;
\end{equation*}

\begin{multline*}
\{1,t\}\cap A=\{1,t\}\ \ \ \Longrightarrow\ \ \ \\ \boldsymbol{x}({}_{-A}\mathrm{T}^{(+)})=
\boldsymbol{\sigma}(1)-(\;\underbrace{\boldsymbol{\sigma}(1)-\boldsymbol{\sigma}(2)}_{2\cdot(\mathbf{M}^{-1})_1}\;)-
(\;\underbrace{\boldsymbol{\sigma}(1)+\boldsymbol{\sigma}(t)}_{2\cdot(\mathbf{M}^{-1})_t}\;)\;
-\sum_{i\in A-\{1,t\}}(\;\underbrace{\boldsymbol{\sigma}(i)-\boldsymbol{\sigma}(i+1)}_{2\cdot(\mathbf{M}^{-1})_i}\;)\; ;
\end{multline*}

\begin{equation*}
|\{1,t\}\cap A|=0\ \ \ \Longrightarrow\ \ \ \boldsymbol{x}({}_{-A}\mathrm{T}^{(+)})=\boldsymbol{\sigma}(1)-\sum_{i\in A}(\;\underbrace{\boldsymbol{\sigma}(i)-\boldsymbol{\sigma}(i+1)}_{2\cdot(\mathbf{M}^{-1})_i}\;)\; ;
\end{equation*}

\begin{equation*}
\{1,t\}\cap A=\{t\}\ \ \ \Longrightarrow\ \ \ \boldsymbol{x}({}_{-A}\mathrm{T}^{(+)})=
\boldsymbol{\sigma}(1)-(\;\underbrace{\boldsymbol{\sigma}(1)+\boldsymbol{\sigma}(t)}_{2\cdot(\mathbf{M}^{-1})_t}\;)\;
-\sum_{i\in A-\{t\}}(\;\underbrace{\boldsymbol{\sigma}(i)-\boldsymbol{\sigma}(i+1)}_{2\cdot(\mathbf{M}^{-1})_i}\;)\; .
\end{equation*}
We arrive at the following conclusion:
\begin{remark}
\label{prop:2}
If $A\subseteq E_t$, then
\begin{multline*}
\boldsymbol{x}({}_{-A}\mathrm{T}^{(+)})=-\boldsymbol{x}({}_{-(E_t-A)}\mathrm{T}^{(+)})\\=
\begin{cases}
\phantom{-}\boldsymbol{\sigma}(2)-\sum_{i\in A-\{1\}}(\;\boldsymbol{\sigma}(i)-\boldsymbol{\sigma}(i+1)\;)\; , & \text{if $\{1,t\}\cap A=\{1\}$}\; ,\\
-\boldsymbol{\sigma}(1)+\boldsymbol{\sigma}(2)-\boldsymbol{\sigma}(t)-\sum_{i\in A-\{1,t\}}(\;\boldsymbol{\sigma}(i)-\boldsymbol{\sigma}(i+1)\;)\; , & \text{if $\{1,t\}\cap A=\{1,t\}$}\; ,\\
\phantom{-}\boldsymbol{\sigma}(1)-\sum_{i\in A}(\;\boldsymbol{\sigma}(i)-\boldsymbol{\sigma}(i+1)\;)\; , & \text{if $|\{1,t\}\cap A|=0$}\; ,\\
-\boldsymbol{\sigma}(t)-\sum_{i\in A-\{t\}}(\;\boldsymbol{\sigma}(i)-\boldsymbol{\sigma}(i+1)\;)\; , & \text{if $\{1,t\}\cap A=\{t\}$}\; .
\end{cases}
\end{multline*}
\end{remark}

Since the sums appearing in Remark~\ref{prop:2} depend only on the endpoints of intervals that compose the sets $A$, we obtain the following explicit descriptions of the decompositions of vertices of the hypercube graph~$\boldsymbol{H}(t,2)$:

\begin{proposition}
\label{prop:3}
Let $\boldsymbol{R}$ be the symmetric cycle in the hypercube graph~$\boldsymbol{H}(t,2)$, defined by~{\rm(\ref{eq:9})(\ref{eq:10})}.

Let $A$ be a nonempty subset of $E_t$, and let
\begin{equation}
\label{eq:16}
A=[i_1,j_1]\;\dot\cup\;[i_2,j_2]\;\dot\cup\;\cdots\;\dot\cup\;[i_{\varrho},j_{\varrho}]
\end{equation}
be its partition into intervals such that
\begin{equation}
\label{eq:17}
j_1+2\leq i_2,\ \ j_2+2\leq i_3,\ \ \ldots,\ \
j_{\varrho-1}+2\leq i_{\varrho}\; ,
\end{equation}
for some $\varrho:=\varrho(A)$.

\begin{itemize}
\item[\rm(i)]
If $\{1,t\}\cap A=\{1\}$, then
\begin{align*}
|\boldsymbol{Q}({}_{-A}\mathrm{T}^{(+)},\boldsymbol{R})|&=
2\varrho-1\; ,\\
\boldsymbol{x}({}_{-A}\mathrm{T}^{(+)},\boldsymbol{R})&=
\sum_{1\leq k\leq\varrho}\boldsymbol{\sigma}(j_k+1)-\sum_{2\leq \ell\leq\varrho}\boldsymbol{\sigma}(i_{\ell})\; .
\end{align*}

\item[\rm(ii)]
If $\{1,t\}\cap A=\{1,t\}$, then
\begin{align*}
|\boldsymbol{Q}({}_{-A}\mathrm{T}^{(+)},\boldsymbol{R})|&=2\varrho-1\; ,\\
\boldsymbol{x}({}_{-A}\mathrm{T}^{(+)},\boldsymbol{R})&=-\boldsymbol{\sigma}(1)+\sum_{1\leq k\leq\varrho-1}\boldsymbol{\sigma}(j_k+1)-
\sum_{2\leq\ell \leq\varrho}\boldsymbol{\sigma}(i_{\ell})\; .
\end{align*}

\item[\rm(iii)]
If $|\{1,t\}\cap A|=0$, then
\begin{align*}
|\boldsymbol{Q}({}_{-A}\mathrm{T}^{(+)},\boldsymbol{R})|&=2\varrho\;+1\; ,\\
\boldsymbol{x}({}_{-A}\mathrm{T}^{(+)},\boldsymbol{R})&=\boldsymbol{\sigma}(1)\;+\sum_{1\leq k\leq\varrho}\boldsymbol{\sigma}(j_k+1)\;-
\sum_{1\leq\ell \leq\varrho}\boldsymbol{\sigma}(i_{\ell})\; .
\end{align*}

\item[\rm(iv)]
If $\{1,t\}\cap A=\{t\}$, then
\begin{align*}
|\boldsymbol{Q}({}_{-A}\mathrm{T}^{(+)},\boldsymbol{R})|&=2\varrho-1\; ,\\
\boldsymbol{x}({}_{-A}\mathrm{T}^{(+)},\boldsymbol{R})&=
\sum_{1\leq k\leq\varrho-1}\boldsymbol{\sigma}(j_k+1)-\sum_{1\leq \ell\leq\varrho}\boldsymbol{\sigma}(i_{\ell})\; .
\end{align*}
\end{itemize}
\end{proposition}

In particular, we have
\begin{equation*}
1\leq j<t\ \ \ \Longrightarrow\ \ \
\boldsymbol{x}({}_{-[j]}\mathrm{T}^{(+)})=\boldsymbol{\sigma}(j+1)\; ;
\end{equation*}
\begin{equation*}
\boldsymbol{x}(\mathrm{T}^{(-)})=-\boldsymbol{\sigma}(1)\; ;
\end{equation*}
\begin{equation*}
1<i<j<t\ \ \ \Longrightarrow\ \ \
\boldsymbol{x}({}_{-[i,j]}\mathrm{T}^{(+)})=\boldsymbol{\sigma}(1)-\boldsymbol{\sigma}(i)+\boldsymbol{\sigma}(j+1)\; ;
\end{equation*}
\begin{equation*}
1<i\leq t\ \ \ \Longrightarrow\ \ \
\boldsymbol{x}({}_{-[i,t]}\mathrm{T}^{(+)})=-\boldsymbol{\sigma}(i)\; .
\end{equation*}

Note that for any vertex $T$ of $\boldsymbol{H}(t,2)$ and for elements $e\in E_t$, we have
\begin{equation*}
x_e(T)\neq 0\ \ \ \Longrightarrow\ \ \ x_e(T)=T(e)\; .
\end{equation*}

Let $\mathtt{c}(m;n)$ denote the number of {\em compositions\/} of a positive integer $n$ with $m$ positive parts.

\begin{remark}
\begin{itemize}
\item[\rm(i)]
It follows directly from the basic enumerative result on {\em compositions of integers} {\rm(}see, e.g.,~\cite[p.~17]{Bona}, \cite[Theorem~1.3]{HM}, \cite[p.~18]{St}{\rm)} that
there are precisely $2\mathtt{c}(2\varrho;t)=2\tbinom{t-1}{2\varrho-1}$ subsets $A\subseteq E_t$, with their partitions {\rm(\ref{eq:16})}{\rm(\ref{eq:17})} into intervals, such that \mbox{$|\{1,t\}\cap A|=1$.}
\item[\rm(ii)] There are $\mathtt{c}(2\varrho-1;t)=\tbinom{t-1}{2(\varrho-1)}$ subsets $A\subseteq E_t$, with their partitions~{\rm(\ref{eq:16})}{\rm(\ref{eq:17})}, such that $|\{1,t\}\cap A|=2$.
\item[\rm(iii)] There are $\mathtt{c}(2\varrho+1;t)=\tbinom{t-1}{2\varrho}$ subsets $A\subseteq E_t$, with their partitions~{\rm(\ref{eq:16})}{\rm(\ref{eq:17})}, such that $|\{1,t\}\cap A|=0$.
\end{itemize}
\end{remark}

Recall that for an odd integer $\ell\in E_t$, there are $2\tbinom{t}{\ell}$ vertices $T$ of the hypercube graph $\boldsymbol{H}(t,2)$ such that
$|\boldsymbol{Q}(T,\boldsymbol{R})|=\ell$; see~\cite[Th.~13.6]{AM-PROM-I}.

\begin{lemma}
Let $\ell\in [3,t]$ be an odd integer. Consider the symmetric cycle~$\boldsymbol{R}$ in the hypercube graph~$\boldsymbol{H}(t,2)$, defined by~{\rm(\ref{eq:9})(\ref{eq:10})}, and the subset of vertices
\begin{equation}
\label{eq:20}
\{T\in\{1,-1\}^t\colon |\boldsymbol{Q}(T,\boldsymbol{R})|=\ell\}\; .
\end{equation}
\begin{itemize}
\item[\rm(i)] {\rm(a)} In the set~{\rm(\ref{eq:20})} there are $\tbinom{t-1}{\ell}$ topes $T$ whose negative parts $T^-$ are disjoint unions
\begin{multline}
\label{eq:21}
T^-=[i_1,j_1]\;\dot\cup\;[i_2,j_2]\;\dot\cup\;\cdots\;\dot\cup\;[i_{(\ell+1)/2},j_{(\ell+1)/2}]\; ;\\
j_1+2\leq i_2,\ \ j_2+2\leq i_3,\ \ \ldots,\ \ j_{(\ell-1)/2}+2\leq i_{(\ell+1)/2}\; ,
\end{multline}
of $\tfrac{\ell+1}{2}$ intervals of $E_t$, and
\begin{equation*}
\{1,t\}\cap T^-=\{i_1\}=\{1\}\; .
\end{equation*}

More precisely, if
\begin{equation*}
\tfrac{\ell+1}{2}\leq j\leq t-\tfrac{\ell+1}{2}\; ,
\end{equation*}
then in the set~{\rm(\ref{eq:20})} there are
\begin{equation*}
\mathtt{c}(\tfrac{\ell+1}{2};j)\cdot\mathtt{c}(\tfrac{\ell+1}{2};t-j)=
\tbinom{j-1}{(\ell-1)/2}\tbinom{t-j-1}{(\ell-1)/2}
\end{equation*}
topes\hfill $T$\hfill whose\hfill negative\hfill parts\hfill $T^-$,\hfill of\hfill cardinality\hfill $j$,\hfill are\hfill disjoint\newline unions~{\rm(\ref{eq:21})} of $\tfrac{\ell+1}{2}$ intervals of $E_t$, such that \mbox{$\{1,t\}\cap T^-=\{i_1\}=\{1\}$.}

\noindent{\em(b)} In the set~{\rm(\ref{eq:20})} there are $\tbinom{t-1}{\ell}$ topes $T$ whose negative parts~$T^-$ are disjoint unions~{\rm(\ref{eq:21})}
of $\tfrac{\ell+1}{2}$ intervals of $E_t$, and
\begin{equation*}
\{1,t\}\cap T^-=\{j_{(\ell+1)/2}\}=\{t\}\; .
\end{equation*}
More precisely, if $\tfrac{\ell+1}{2}\leq j\leq t-\tfrac{\ell+1}{2}$, then in the set~{\rm(\ref{eq:20})} there are~$\tbinom{j-1}{(\ell-1)/2}\tbinom{t-j-1}{(\ell-1)/2}$ topes $T$ whose negative parts~$T^-$, with $|T^-|=j$, are disjoint unions~{\rm(\ref{eq:21})} of $\tfrac{\ell+1}{2}$ intervals of $E_t$, such that~$\{1,t\}\cap T^-$ $=\{j_{(\ell+1)/2}\}=\{t\}$.

\item[\rm(ii)]
In the set~{\rm(\ref{eq:20})} there are $\tbinom{t-1}{\ell-1}$ topes $T$ whose negative parts $T^-$ are disjoint unions~{\rm(\ref{eq:21})}
of $\tfrac{\ell+1}{2}$ intervals of $E_t$, and
\begin{equation*}
\{1,t\}\cap T^-=\{i_1,j_{(\ell+1)/2}\}=\{1,t\}\; .
\end{equation*}

More precisely, if
\begin{equation*}
\tfrac{\ell+1}{2}\leq j\leq t-\tfrac{\ell-1}{2}\; ,
\end{equation*}
then in the set~{\rm(\ref{eq:20})} there are
\begin{equation*}
\mathtt{c}(\tfrac{\ell+1}{2};j)\cdot\mathtt{c}(\tfrac{\ell-1}{2};t-j)=
\tbinom{j-1}{(\ell-1)/2}\tbinom{t-j-1}{(\ell-3)/2}
\end{equation*}
topes $T$ whose negative parts $T^-$, of cardinality $j$, are disjoint unions~{\rm(\ref{eq:21})}
of $\tfrac{\ell+1}{2}$ intervals of $E_t$, such that $\{1,t\}\cap T^-=\{i_1,j_{(\ell+1)/2}\}=\{1,t\}$.

\item[\rm(iii)] In the set~{\rm(\ref{eq:20})} there are $\tbinom{t-1}{\ell-1}$ topes $T$ whose negative parts $T^-$ are disjoint unions
\begin{multline}
\label{eq:41}
T^-=[i_1,j_1]\;\dot\cup\;[i_2,j_2]\;\dot\cup\;\cdots\;\dot\cup\;[i_{(\ell-1)/2},j_{(\ell-1)/2}]\; ;\\
j_1+2\leq i_2,\ \ j_2+2\leq i_3,\ \ \ldots,\ \ j_{(\ell-3)/2}+2\leq i_{(\ell-1)/2}\; ,
\end{multline}
of $\tfrac{\ell-1}{2}$ intervals of $E_t$, and
\begin{equation*}
|\{1,t\}\cap T^-|=0\; .
\end{equation*}

More precisely, if
\begin{equation*}
\tfrac{\ell-1}{2}\leq j\leq t-\tfrac{\ell+1}{2}\; ,
\end{equation*}
then in the set~{\rm(\ref{eq:20})} there are
\begin{equation*}\mathtt{c}(\tfrac{\ell-1}{2};j)\cdot\mathtt{c}(\tfrac{\ell+1}{2};t-j)=
\tbinom{j-1}{(\ell-3)/2}\tbinom{t-j-1}{(\ell-1)/2}
\end{equation*}
topes $T$ whose negative parts $T^-$, of cardinality~$j$, are disjoint unions~{\rm(\ref{eq:41})} of $\tfrac{\ell-1}{2}$ intervals of $E_t$, such that $|\{1,t\}\cap T^-|=0$.
\end{itemize}
\end{lemma}

We can now give a refined statistic on the decompositions of vertices with respect to the distinguished symmetric cycle.
\begin{theorem}
Let $j\in E_t$, and let $\ell\in[3,t]$ be an odd integer. Consider the symmetric cycle~$\boldsymbol{R}$ in the hypercube graph~$\boldsymbol{H}(t,2)$, defined by~{\rm(\ref{eq:9})(\ref{eq:10})}.
\begin{itemize}
\item[\rm(i)]
If
\begin{equation*}
j<\tfrac{\ell-1}{2}\ \ \ \text{or}\ \ \ j>t-\tfrac{\ell-1}{2}\; ,
\end{equation*}
then
\begin{equation*}
|\{T\in\{1,-1\}^t\colon |T^-|=j,\ |\boldsymbol{Q}(T,\boldsymbol{R})|=\ell\}|=0\; .
\end{equation*}
\item[\rm(ii)]
If
\begin{equation*}
\tfrac{\ell-1}{2}\leq j\leq t-\tfrac{\ell-1}{2}\; ,
\end{equation*}
then
\begin{gather*}
|\{T\in\{1,-1\}^t\colon |T^-|=j,\ |\boldsymbol{Q}(T,\boldsymbol{R})|=\ell\}|\\=
|\{T\in\{1,-1\}^t\colon |T^-|=t-j,\ |\boldsymbol{Q}(T,\boldsymbol{R})|=\ell\}|\\=
2\mathtt{c}(\tfrac{\ell+1}{2};j)\cdot\mathtt{c}(\tfrac{\ell+1}{2};t-j)+\mathtt{c}(\tfrac{\ell+1}{2};j)\cdot\mathtt{c}(\tfrac{\ell-1}{2};t-j)
+\mathtt{c}(\tfrac{\ell-1}{2};j)\cdot\mathtt{c}(\tfrac{\ell+1}{2};t-j)\\=
\binom{j-1}{\frac{\ell-1}{2}}\binom{t-j}{\frac{\ell-1}{2}}+
\binom{t-j-1}{\frac{\ell-1}{2}}\binom{j}{\frac{\ell-1}{2}}\\=
\mathtt{c}(\tfrac{\ell+1}{2};j)\cdot\mathtt{c}(\tfrac{\ell+1}{2};t-j+1)
+\mathtt{c}(\tfrac{\ell+1}{2};t-j)\cdot\mathtt{c}(\tfrac{\ell+1}{2};j+1)\; .
\end{gather*}
In particular, we have
\begin{equation*}
|\{T\in\{1,-1\}^t\colon |T^-|=\tfrac{\ell-1}{2},\ |\boldsymbol{Q}(T,\boldsymbol{R})|=\ell\}|
=\binom{t-\frac{\ell+1}{2}}{\frac{\ell-1}{2}}=\binom{t-\frac{\ell+1}{2}}{t-\ell}\; ,
\end{equation*}
and for $j\in[t-1]$, we have
\begin{equation*}
|\{T\in\{1,-1\}^t\colon |T^-|=j,\ |\boldsymbol{Q}(T,\boldsymbol{R})|=3\}|=2j(t-j)-t\; .
\end{equation*}
\end{itemize}
\end{theorem}

\section{The sizes of decompositions}

For\hfill the\hfill symmetric\hfill cycle\hfill $\boldsymbol{R}$\hfill in\hfill the\hfill hypercube\hfill graph\hfill $\boldsymbol{H}(t,2)$,\hfill defined\newline by~(\ref{eq:9})(\ref{eq:10}), the $(i,j)$th entry of the symmetric Toeplitz matrix~$\mathbf{M}\cdot\mathbf{M}^{\top}$, where $\mathbf{M}:=\mathbf{M}(\boldsymbol{R})$ is given in~(\ref{eq:26}), is
\begin{equation*}
t-2|j-i|\; ,
\end{equation*}
while\hfill the\hfill $i$th\hfill row\hfill $(\mathbf{M}^{-1}\cdot(\mathbf{M}^{-1})^{\top})_i$\hfill
of\hfill the\hfill symmetric\hfill Toeplitz\hfill matrix\newline $\mathbf{M}^{-1}\cdot(\mathbf{M}^{-1})^{\top}$ is
\begin{equation}
\label{eq:25}
\bigl(\mathbf{M}^{-1}\cdot(\mathbf{M}^{-1})^{\top}\bigr)_i=
\begin{cases}
\frac{1}{4}\cdot(\,2\boldsymbol{\sigma}(1)-\boldsymbol{\sigma}(2)+\boldsymbol{\sigma}(t)\,)\; , & \text{if $i=1$}\; ,\\
\frac{1}{4}\cdot(\,-\boldsymbol{\sigma}(i-1)+2\boldsymbol{\sigma}(i)+\boldsymbol{\sigma}(i+1)\,)\; , & \text{if $2\leq i\leq t-1$}\; ,\\
\frac{1}{4}\cdot(\,\boldsymbol{\sigma}(1)-\boldsymbol{\sigma}(t-1)+2\boldsymbol{\sigma}(t)\,)\; , & \text{if $i=t$}\; .
\end{cases}
\end{equation}
Recall that for a tope $T\in\{1,-1\}^t$ we have
\begin{equation*}
|\boldsymbol{Q}(T,\boldsymbol{R})|=\|\boldsymbol{x}(T)\|^2=
T\cdot\mathbf{M}^{-1}\cdot(\mathbf{M}^{-1})^{\top}\cdot T^{\top}\; .
\end{equation*}
If $T',T''\in\{1,-1\}^t$, then~(\ref{eq:22}) implies that
\begin{multline*}
|\boldsymbol{Q}(T'',\boldsymbol{R})|=|\boldsymbol{Q}(T',\boldsymbol{R})|-
4\Bigl\langle T'\mathbf{M}^{-1},\Bigl(\sum_{s\in\mathbf{S}(T',T'')}T'(s)\boldsymbol{\sigma}(s)\Bigr)\cdot\mathbf{M}^{-1}\Bigr\rangle
\\+4\Bigl\|\Bigl(\sum_{s\in\mathbf{S}(T',T'')}T'(s)\boldsymbol{\sigma}(s)\Bigr)\cdot\mathbf{M}^{-1}\Bigr\|^2\; ,
\end{multline*}
and, as a consequence, we have
\begin{multline}
\label{eq:32}
|\boldsymbol{Q}(T',\boldsymbol{R})|-|\boldsymbol{Q}(T'',\boldsymbol{R})|
\\=
4\Bigl\langle \Bigl(T'-\Bigl(\sum_{s\in\mathbf{S}(T',T'')}T'(s)\boldsymbol{\sigma}(s)\Bigr)\Bigr)\cdot\mathbf{M}^{-1},
\Bigl(\sum_{s\in\mathbf{S}(T',T'')}T'(s)\boldsymbol{\sigma}(s)\Bigr)\cdot\mathbf{M}^{-1}\Bigr\rangle
\; .
\end{multline}

\section{Equinumerous decompositions}

If $T',T''\in\{1,-1\}^t$ are two vertices of the hypercube graph~$\boldsymbol{H}(t,2)$ with its distinguished symmetric cycle~$\boldsymbol{R}$, defined by~(\ref{eq:9})(\ref{eq:10}), then it follows from~(\ref{eq:32}) that
\begin{equation}
\label{eq:24}
|\boldsymbol{Q}(T',\boldsymbol{R})|=|\boldsymbol{Q}(T'',\boldsymbol{R})|
\end{equation}
if and only if
\begin{equation*}
\Bigl\langle \Bigl(T'-\Bigl(\sum_{s\in\mathbf{S}(T',T'')}T'(s)\boldsymbol{\sigma}(s)\Bigr)\Bigr)\cdot\mathbf{M}^{-1},
\Bigl(\sum_{s\in\mathbf{S}(T',T'')}T'(s)\boldsymbol{\sigma}(s)\Bigr)\cdot\mathbf{M}^{-1}\Bigr\rangle=0\; .
\end{equation*}
Let us denote by $\omega(i,j)$ the $(i,j)$th entry of the matrix~$\mathbf{M}^{-1}\cdot(\mathbf{M}^{-1})^{\top}$ whose rows are given in~(\ref{eq:25}). We see that~(\ref{eq:24}) holds if and only if
\begin{equation*}
\begin{split}
\sum_{i\in E_t-\mathbf{S}(T',T'')}\;\sum_{j\in\mathbf{S}(T',T'')}T'(i)\cdot T'(j)\cdot\omega(i,j)&=\\
\sum_{i\in\mathbf{S}(T',T'')}\;\sum_{j\in E_t-\mathbf{S}(T',T'')} T'(i)\cdot T'(j)\cdot\omega(i,j)&=0
\end{split}
\end{equation*}
or, equivalently,
\begin{equation}
\label{eq:31}
\sum_{i\in[t-1]}\; \sum_{\substack{j\in[i+1,t]:\\ |\{i,j\}\cap\mathbf{S}(T',T'')|=1}}
T'(i)\cdot T'(j)\cdot\omega(i,j)=0\; .
\end{equation}

\begin{proposition}
Let $\boldsymbol{R}$ be the symmetric cycle in the hypercube graph~$\boldsymbol{H}(t,2)$, defined by~{\rm(\ref{eq:9})(\ref{eq:10})}.

Let $T\in\{1,-1\}^t$ be a vertex of $\boldsymbol{H}(t,2)$, and let $A$ be a proper subset of the set~$E_t$.

\begin{itemize}
\item[\rm(i)] If $|\{1,t\}\cap A|=1$, then
\begin{equation*}
|\boldsymbol{Q}(T,\boldsymbol{R})|=|\boldsymbol{Q}({}_{-A}T,\boldsymbol{R})|\ \Longleftrightarrow\
\sum_{\substack{i\in[t-1]:\\ |\{i,i+1\}\cap A|=1}}
T(i)\cdot T(i+1)=T(1)\cdot T(t)\; .
\end{equation*}

\item[\rm(ii)] If $|\{1,t\}\cap A|\neq 1$, then
\begin{equation*}
|\boldsymbol{Q}(T,\boldsymbol{R})|=|\boldsymbol{Q}({}_{-A}T,\boldsymbol{R})|\ \Longleftrightarrow\
\sum_{\substack{i\in[t-1]:\\ |\{i,i+1\}\cap A|=1}}
T(i)\cdot T(i+1)=0\; .
\end{equation*}
\end{itemize}
\end{proposition}

\begin{proof} Consider relation~(\ref{eq:31}) for the topes $T':=T$ and $T'':={}_{-A}T'$ with their separation set~$\mathbf{S}(T',T'')=A$.
\begin{itemize}
\item[\rm(i)] Since $|\{1,t\}\cap \mathbf{S}(T',T'')|=1$, we have
\begin{multline*}
|\boldsymbol{Q}(T',\boldsymbol{R})|=|\boldsymbol{Q}(T'',\boldsymbol{R})|\ \Longleftrightarrow\ T'(1)\cdot T'(t)\cdot 1\\+\sum_{i\in[t-1]}\; \sum_{\substack{j\in[i+1,t]:\\ |\{i,j\}\cap\mathbf{S}(T',T'')|=1}}
T'(i)\cdot T'(j)\cdot(-1)\\=
T'(1)\cdot T'(t)\cdot 1+\sum_{\substack{i\in[t-1]:\\ |\{i,i+1\}\cap\mathbf{S}(T',T'')|=1}}
T'(i)\cdot T'(i+1)\cdot(-1)=0\; .
\end{multline*}

\item[\rm(ii)] Since $|\{1,t\}\cap \mathbf{S}(T',T'')|\neq 1$, we have
\begin{multline*}
|\boldsymbol{Q}(T',\boldsymbol{R})|=|\boldsymbol{Q}(T'',\boldsymbol{R})|\ \Longleftrightarrow\ \sum_{i\in[t-1]}\; \sum_{\substack{j\in[i+1,t]:\\ |\{i,j\}\cap\mathbf{S}(T',T'')|=1}}
T'(i)\cdot T'(j)\cdot(-1)\\=
\sum_{\substack{i\in[t-1]:\\ |\{i,i+1\}\cap\mathbf{S}(T',T'')|=1}}
T'(i)\cdot T'(i+1)\cdot(-1)=0\; . \qedhere
\end{multline*}
\end{itemize}
\end{proof}

We conclude this note with simple structural criteria (derived from~Proposition~\ref{prop:3}) of the equicardinality of decompositions.
\begin{corollary}
Let $\boldsymbol{R}$ be the symmetric cycle in the hypercube graph~$\boldsymbol{H}(t,2)$, defined by~{\rm(\ref{eq:9})(\ref{eq:10})}.

Let $A$ and $B$ be two nonempty subsets of the set $E_t$, and let
\begin{align*}
A&=[i'_1,j'_1]\;\dot\cup\;[i'_2,j'_2]\;\dot\cup\;\cdots\;\dot\cup\;[i'_{\varrho(A)},j'_{\varrho(A)}]
\intertext{and}
B&=[i''_1,j''_1]\;\dot\cup\;[i''_2,j''_2]\;\dot\cup\;\cdots\;\dot\cup\;[i''_{\varrho(B)},j''_{\varrho(B)}]
\end{align*}
be their partitions into intervals such that
\begin{equation*}
j'_1+2\leq i'_2,\ \ j'_2+2\leq i'_3,\ \ \ldots,\ \
j'_{\varrho(A)-1}+2\leq i'_{\varrho(A)}
\end{equation*}
and
\begin{equation*}
j''_1+2\leq i''_2,\ \ j''_2+2\leq i''_3,\ \ \ldots,\ \
j''_{\varrho(B)-1}+2\leq i''_{\varrho(B)}\; .
\end{equation*}

\begin{itemize}
\item[\rm(i)]
If
\begin{equation*}
|\{1,t\}\cap A|>0\ \ \ \text{and}\ \ \ |\{1,t\}\cap B|>0\; ,
\end{equation*}
or
\begin{equation*}
|\{1,t\}\cap A|=|\{1,t\}\cap B|=0\; ,
\end{equation*}
then
\begin{equation*}
\begin{split}
|\boldsymbol{Q}({}_{-A}\mathrm{T}^{(+)},\boldsymbol{R})|
=|\boldsymbol{Q}({}_{-B}\mathrm{T}^{(+)},\boldsymbol{R})|\ \ \ \Longleftrightarrow\ \ \ \varrho(B)=
\varrho(A)\; .
\end{split}
\end{equation*}

\item[\rm(ii)]
If
\begin{equation*}
|\{1,t\}\cap A|>0\ \ \ \text{and}\ \ \ |\{1,t\}\cap B|=0\; ,
\end{equation*}
then
\begin{equation*}
|\boldsymbol{Q}({}_{-A}\mathrm{T}^{(+)},\boldsymbol{R})|
=|\boldsymbol{Q}({}_{-B}\mathrm{T}^{(+)},\boldsymbol{R})|\ \ \ \Longleftrightarrow\ \ \ \varrho(B)=
\varrho(A)-1\; .
\end{equation*}
\end{itemize}
\end{corollary}

\vspace{5mm}


\begin{thebibliography}{99.}
\bibitem{BLSWZ}
{\em Bj\"{o}rner A.}, {\em Las~Vergnas M.}, {\em Sturmfels~B.}, {\em White~N.}, {\em Ziegler G.M.} Oriented matroids. Second
edition. Encyclopedia of Mathematics, 46. Cambridge: Cambridge University Press, 1999.

\bibitem{Bona}
{\em B\'{o}na~M.} ({\em ed.}) Handbook of enumerative combinatorics. Discrete Mathematics and its Applications (Boca Raton). Boca~Raton, FL: CRC~Press, 2015.

\bibitem{HM}
{\em Heubach~S.}, {\em Mansour~T.} Combinatorics of compositions and words. Discrete Mathematics and its Applications (Boca Raton).  Boca~Raton, FL: CRC~Press, 2010.

\bibitem{AM-PROM-I}
{\em Matveev~A.O.} Pattern recognition on oriented matroids. Berlin: De~Gruyter, 2017.

\bibitem{St}
{\em Stanley R.P.} Enumerative combinatorics. Volume~1. Second edition. Cambridge Studies in
Advanced Mathematics, 49. Cambridge: Cambridge University Press, 2012.
\end{thebibliography}
\end{document}